\documentclass[12pt, A4paper]{article}
\usepackage[utf8]{inputenc}
\usepackage[T2A]{fontenc}
\usepackage{amsmath}
\usepackage{amsfonts}
\usepackage{amssymb}
\usepackage{graphicx}
\usepackage[export]{adjustbox}
\newenvironment{proof}{{\bf Proof.}}{\par\hspace{25em}\rule{1ex}{1ex}\par}
\newtheorem{theorem}{Theorem}[section]
\newtheorem{lemma}{Lemma}[section]
\title{Caustics of wave fronts reflected by a surface\footnote{Journal of Mathematical Sciences and Modeling,  2018, V 1, Issue 2, Pages 131 - 137.  https://doi.org/10.33187/jmsm.431543} }
\author{Alexander Yampolsky ${ }^{1}$ and Oleksandr Fursenko ${ }^{2 }$}

\date{}

\begin{document}
\maketitle
\begin{center}
\noindent ${ }^1$  V.N. Karazin Kharkiv National University, Kharkiv, Ukraine\\
${ }^2$  I. Kozhedub National Air Force University, Kharkiv, Ukraine\\[1ex]
\end{center}
*Corresponding author E-mail: a.yampolsky@karazin.ua, \\
https: //orcid.org/0000-0002-7215-3669\\[1ex]
2010 AMS: 53A25, 53B20, 53B25

\begin{abstract}
One can often see caustic by reflection in nature but it is rather hard to understand the way of how caustic arise and which geometric properties of a mirror surface define geometry of the caustic. The caustic by reflection has complicated topology and much more complicated geometry. From engineering point of view the geometry of caustic by reflection is important for antenna's theory because it can be considered as a surface of concentration of the reflected wave front. In this paper we give purely geometric description of the caustics of wave front (flat or spherical) after reflection from mirror surface. The description clarifies the dependence of caustic on geometrical characteristics of a surface and allows rather simple and fast computer visualization of the caustics in dependence of location of the rays source or direction of the pencil of parallel rays.
\end{abstract}

\section*{Introduction}
The caustic of reflected wave front is the envelope of the family of rays emitted from a given point $O$ (source) and reflected by smooth surface $S$ (mirror). This type of caustics are called by catacaustics or caustics by reflection. There are a number of research on caustics and catacaustics. Mostly, the authors were interested in topology of the caustic \cite{1,3}.  Nice research concerning location of the catacaustic inside compact convex body one can find in \cite{2}. The catacaustics of a canal surfaces were studied in \cite{4}. The caustics of a translation surfaces were considered in \cite{6}. The caustics on spheres and cylinders of revolution were studied in \cite{7}. A survey on the theory of caustics and wave front propagations with applications to geometry one can found in \cite{5}. The catacaustics from engineering viewpoint in connection with antenna's technologies have been studied in \cite{8}. The research presented in the mentioned above papers do not use the parametric equation of the caustic.

In this paper, we show that to find the caustic of reflected rays of a surface we should consider a "virtual deformation" of the mirror and its second fundamental form under influence of the incident pencil of rays. This approach allowed to give an explicit parameterization of the caustic of the reflected front in geometric terms of the mirror surface (Theorem \textbf{1.1}, Theorem \textbf{2.1}).

\section{ Caustics by reflected in case of flat incident front}
Let us be given a regular surface $S$ parameterized by vector-function $\mathbf{r}: D^{2}\left(u^{1}, u^{2}\right) \rightarrow S \subset E^{3}$. Let $\mathbf{n}$ be the unit normal vector field of S oriented in such a way that the support function is positive (the origin is in positive half-space with respect to each tangent plane). Suppose $\mathbf{a}(|\mathbf{a}|=1)$ is a unit normal to a flat front incident on the surface and such that the dot-product $(\mathbf{a}, \mathbf{n})<0$. If $\mathbf{b}$ is the direction of reflected ray, then
\begin{equation*}
\mathbf{b}=\mathbf{a}-2(\mathbf{a}, \mathbf{n}) \mathbf{n} \tag{1.1}
\end{equation*}
We also suppose that at initial time $t=0$ the incident front $F_{0}$ passes through the origin. Let $L$ be the distance that each "photon" of $F_{0}$ passes along the ray in direction of $\mathbf{a}$. The front $F_{t}$ reaches the surface $S$ if $L \geq(\mathbf{r}, \mathbf{a})$. Denote by $\lambda$ the distance that each "photon" of the front $F_{t}$ passes after reflection. Then evidently $(\mathbf{r}, \mathbf{a})+\lambda=L$ and hence
\begin{equation*}
\lambda=L-(\mathbf{r}, \mathbf{a}) \tag{1.2}
\end{equation*}
Therefore, the parametric equation of the reflected front is
\begin{equation*}
\boldsymbol\rho=\mathbf{r}+\lambda \cdot \mathbf{b} \quad(\lambda \geq 0) \tag{1.3}
\end{equation*}
The following assertion simplifies calculations.
\begin{lemma} Denote by $g_{i j}$ and $B_{i j}$ the matrices of the first and the second fundamental forms of a regular surface $S \subset E^{3}$, respectively. $A$ caustic of the reflected front of the surface $S$ is the same as a focal set of a surface with the first and second fundamental forms given by
\begin{equation*}
g_{i j}^{*}=g_{i j}-\left(\partial_{j} \mathbf{r}, \mathbf{a}\right)\left(\partial_{i} \mathbf{r}, \mathbf{a}\right), \quad B_{i j}^{*}=-2 \cos \theta B_{i j} . \tag{1.4}
\end{equation*}
\end{lemma}
\begin{proof}  Let
\begin{equation*}
\boldsymbol\rho=\mathbf{r}+\lambda \cdot \mathbf{b} \quad(\lambda \geq 0) \tag{1.5}
\end{equation*}
be a parametric equation of the reflected front, where $\mathbf{b}$ is given by (1.1). Let us show that $\mathbf{b}$ is a unit normal vector field of the reflected front. Indeed,
\begin{equation*}
\partial_{i} \boldsymbol\rho=\partial_{i} \mathbf{r}+\partial_{i} \lambda \cdot \mathbf{b}+\lambda \cdot \partial_{i} \mathbf{b} \tag{1.6}
\end{equation*}
and hence $\left(\partial_{i} \boldsymbol\rho, \mathbf{b}\right)=\left(\partial_{i} \mathbf{r}, \mathbf{b}\right)+\partial_{i} \lambda=\left(\partial_{i} \mathbf{r}, \mathbf{b}\right)-\left(\partial_{i} \mathbf{r}, \mathbf{a}\right)=0$. Let $g$ and $B$ are the first and the second fundamental forms of $S$. Denote by $g^{*}$ and $B^{*}$ the first and the second fundamental forms of the reflected front. Denote, in addition, $b_{i j}^{*}=\left(\partial_{i} \mathbf{b}, \partial_{j} \mathbf{b}\right)$ and $\cos \theta=(\mathbf{a}, \mathbf{n})$. Then
$$
g_{i j}^{*}=g_{i j}+4 \lambda \cos \theta B_{i j}+\lambda^{2} b_{i j}^{*}-\left(\partial_{i} \mathbf{r}, \mathbf{a}\right)\left(\partial_{j} \mathbf{r}, \mathbf{a}\right), \quad B_{i j}^{*}=-2 \cos \theta B_{i j}+\lambda b_{i j}^{*}
$$
To prove this, remark that $g_{i j}^{*}=\left(\partial_{i} \boldsymbol\rho, \partial_{j} \boldsymbol\rho\right)$. From (1.1) and (1.2) it follows that
$$
\partial_{i} \mathbf{b}=-2\left(\mathbf{a}, \partial_{i} \mathbf{n}\right) \mathbf{n}-2(\mathbf{a}, \mathbf{n}) \partial_{i} \mathbf{n}, \quad \partial_{i} \lambda=-\left(\partial_{i} \mathbf{r}, \mathbf{a}\right), \quad \partial_{i} \boldsymbol\rho=\partial_{i} \mathbf{r}-\left(\partial_{i} \mathbf{r}, \mathbf{a}\right) \mathbf{b}+\lambda \partial_{i} \mathbf{b}
$$
Thus we have
\begin{multline*}
 B_{i j}^{*}=-\left(\partial_{i} \boldsymbol\rho, \partial_{j} \mathbf{b}\right)=\\
 -\left(\partial_{i} \mathbf{r}+\partial_{i} \lambda \cdot \mathbf{b}+\lambda \cdot \partial_{i} \mathbf{b}, \partial_{j} \mathbf{b}\right)=
 -\left(\partial_{i} \mathbf{r}, \partial_{j} \mathbf{b}\right)-\lambda b_{i j}^{*}=
  -2(\mathbf{a}, \mathbf{n}) B_{i j}-\lambda b_{i j}^{*},
\end{multline*}
\begin{multline*}
 g_{i j}^{*}=\left(\partial_{i} \boldsymbol\rho, \partial_{j} \boldsymbol\rho\right)= \left(\partial_{i} \mathbf{r}-\left(\partial_{i} \mathbf{r}, \mathbf{a}\right) \mathbf{b}+\lambda \partial_{i} \mathbf{b}, \partial_{j} \rho\right)=
 \left(\partial_{i} \mathbf{r}, \partial_{j} \rho\right)+\lambda\left(\partial_{i} \mathbf{b}, \partial_{j} \boldsymbol\rho\right)=\\
  g_{i j}-\left(\partial_{j} \mathbf{r}, \mathbf{a}\right)\left(\partial_{i} \mathbf{r}, \mathbf{a}\right)+2 \lambda(\mathbf{a}, \mathbf{n}) B_{i j}-\lambda B_{i j}^{*}
\end{multline*}

The caustic of reflected front is nothing else but its focal surface. The latter is formed by striction lines of ruled surface generated by the unit normal vector field along lines of curvature of the reflected front. It is well known that Gaussian curvature of the latter ruled surface has to be zero. If $\mathbf{X}=X^{1} \partial_{1} \boldsymbol\rho+X^{2} \partial_{2} \boldsymbol\rho$ is tangent to line of curvature of the reflected front, then the condition on Gaussian curvature takes the form
\begin{equation*}
\left(\partial_{X} \boldsymbol\rho, \partial_{X} \mathbf{b}, \mathbf{b}\right)=\left(\partial_{X} \mathbf{r}-(X, \mathbf{a}) \mathbf{b}+\lambda \partial_{X} \mathbf{b}, \partial_{X} \mathbf{b}, \mathbf{b}\right)=\left(\partial_{X} \mathbf{r}, \partial_{X} \mathbf{b}, \mathbf{b}\right)=0 . \tag{1.7}
\end{equation*}

The condition (1.7) does not depend on $\lambda$ and is equivalent to the condition on $\mathbf{X}$ to be tangent to principal direction of the reflected front at the initial moment, when the reflected front and the mirror surface coincide pointwise, i.e. when $\lambda=0$. As a consequence, the caustic can be found as a focal surface of reflected front at the moment $\lambda=0$. Therefore, to find the caustic we can take
$$
g_{i j}^{*}=g_{i j}-\left(\partial_{j} \mathbf{r}, \mathbf{a}\right)\left(\partial_{i} \mathbf{r}, \mathbf{a}\right), \quad B_{i j}^{*}=-2 \cos \theta B_{i j},
$$
where $\cos \theta=(\mathbf{a}, \mathbf{n})$ is the angle function between the incident rays and the unit normal vector field of the surface.
\end{proof}

\begin{theorem} Let $S$ be a regular surface of non-zero Gaussian curvature parameterized by position-vector $\mathbf{r}$ and $\mathbf{n}$ be a the unit normal vector field of S. Denote by $\mathbf{a}(|\mathbf{a}|=1)$ a direction of incident rays. Then there exist two caustics of the reflected front and their parametric equations can be given by
$$
\boldsymbol\xi^{*}=\mathbf{r}+\frac{1}{k_{i}^{*}} \mathbf{b} \quad(i=1,2),
$$
where $\mathbf{b}=\mathbf{a}-2(\mathbf{a}, \mathbf{n}) \mathbf{n}$ is a direction of reflected rays and $k_{i}^{*}$ are the roots of the equation
$$
\left(k^{*}\right)^{2}+2 \cos \theta\left(2 H+k_{n}\left(a_{t}\right) \tan ^{2} \theta\right) k^{*}+4 K=0,
$$
where $H$ is the mean curvature, $K$ is the Gaussian curvature and $k_{n}\left(a_{t}\right)$ is the normal curvature of the surface in a direction of tangential projection of the incident rays.
\end{theorem}

\begin{proof} Introduce on the surface the curvature coordinates $\left(u^{1}, u^{2}\right)$. Then
$$
g=\left(\begin{array}{cc}
g_{11} & 0 \\
0 & g_{22}
\end{array}\right), \quad B=\left(\begin{array}{cc}
k_{1} g_{11} & 0 \\
0 & k_{2} g_{22}
\end{array}\right)
$$
where $k_{1}$ and $k_{2}$ are the principal curvatures of the surface $S$. Introduce the orthonormal frame
\begin{equation*}
\mathbf{e}_{1}=\frac{1}{\sqrt{g_{11}}} \partial_{1} \mathbf{r}, \quad \mathbf{e}_{2}=\frac{1}{\sqrt{g_{22}}} \partial_{2} \mathbf{r}, \quad \mathbf{e}_{3}=\mathbf{n} \tag{1.8}
\end{equation*}
Decompose $\mathbf{a}=a_{1} \mathbf{e}_{1}+a_{2} \mathbf{e}_{2}+a_{3} \mathbf{e}_{3} \quad\left(a_{1}^{2}+a_{2}^{2}+a_{3}^{2}=1\right)$. We have
$$
B^{*}=-2 a_{3}\left(\begin{array}{cc}
k_{1} g_{11} & 0 \\
0 & k_{2} g_{22}
\end{array}\right), \quad \quad g^{*}=\left(\begin{array}{cc}
g_{11}\left(1-a_{1}^{2}\right) & -\sqrt{g_{11} g_{22}} \,a_{1} a_{2} \\
-\sqrt{g_{11} g_{22}}\, a_{1} a_{2} & g_{22}\left(1-a_{2}^{2}\right)
\end{array}\right)
$$
Then, evidently,
$$
\operatorname{det} g^{*}=g_{11} g_{22}\left(1-a_{1}^{2}-a_{2}^{2}\right)=g_{11} g_{22} a_{3}^{2} .
$$
The Weingarten matrix is
$$
A^{*}=-\frac{2}{a_{3}}\left(\begin{array}{cc}
\left(1-a_{2}^{2}\right) k_{1} & a_{1} a_{2} k_{2} \sqrt{\frac{g_{22}}{g_{11}}} \\
a_{1} a_{2} k_{1} \sqrt{\frac{g_{11}}{g_{22}}} & \left(1-a_{1}^{2}\right) k_{2}
\end{array}\right)
$$
Thus, the characteristic equation on $k_{i}^{*}$ takes the form
$$
\left(k^{*}\right)^{2}+\frac{2}{a_{3}}\left(\left(a_{1}^{2}+a_{3}^{2}\right) k_{1}+\left(a_{2}^{2}+a_{3}^{2}\right) k_{2}\right) k^{*}+4 k_{1} k_{2}=0
$$
or
$$
\left(k^{*}\right)^{2}+2 a_{3}\left(k_{1}+k_{2}+\frac{a_{1}^{2} k_{1}+a_{2}^{2} k_{2}}{a_{3}^{2}}\right) k^{*}+4 k_{1} k_{2}=0
$$
It remains to notice that $k_{1}+k_{2}=2 H, k_{1} k_{2}=K, a_{1}^{2} k_{1}+a_{2}^{2} k_{2}=B\left(a_{t}, a_{t}\right)=k_{n}\left(a_{t}\right) g\left(a_{t}, a_{t}\right)=k_{n}\left(a_{t}\right) \sin ^{2} \theta$ and $a_{3}=\cos \theta$. As a result, we obtain
$$
\left(k^{*}\right)^{2}+2 \cos \theta\left(2 H+k_{n}\left(a_{t}\right) \tan ^{2} \theta\right) k^{*}+4 K=0 .
$$
If we denote by $k_{i}^{*}$ the roots, then the equation of caustic of the reflected front takes the form
$$
\boldsymbol\xi_{i}^{*}=\mathbf{r}+\frac{1}{k_{i}^{*}} \mathbf{b} .
$$
Evidently, $k_{1}^{*} k_{2}^{*}=4 K$ and hence, if $K \neq 0$, then both of caustics exit.
\end{proof}

\textbf{Example 1.3.} The sphere of radius 1 . We have $k_{1}=k_{2}=1, K=1, H=1, k_{n}\left(a_{t}\right)=1$ and hence the equation on $k^{*}$ takes the form

$$
\left(k^{*}\right)^{2}+2\left(\cos \theta+\frac{1}{\cos \theta}\right) k^{*}+4=0
$$
with evident solutions $k_{1}^{*}=-2 \cos \theta, \quad k_{2}^{*}=-\frac{2}{\cos \theta}$ So, the parametric equations of caustics of the reflected front take the forms
$$
\boldsymbol\xi_{1}^{*}=\mathbf{r}-\frac{1}{2 \cos \theta} \mathbf{b}, \quad \boldsymbol\xi_{2}^{*}=\mathbf{r}-\frac{\cos \theta}{2} \mathbf{b} .
$$
Take a local parameterization for the sphere as $\mathbf{r}=\{\cos u \cos v, \cos u \sin v, \sin u\}$ and suppose $\mathbf{a}=\{0,0,1\}$. Then $\cos \theta=(\mathbf{a}, \mathbf{n})=-\sin u$ where $u \in(0, \pi / 2)$. Thus, we have two caustics
$$
\boldsymbol\xi_{1}^{*}=\left\{0,0, \frac{1}{2 \sin u}\right\}, \quad \boldsymbol\xi_{2}^{*}=\left\{\cos ^{3} u \cos v, \cos ^{3} u \sin v, \frac{\sin u}{2}\left(2 \cos ^{2} u+1\right)\right\}
$$
The first caustic degenerates into a straight line, the second one is a surface of revolution generated by the caustic of plane circle (see Figure 1.1)
\begin{center}
\includegraphics[max width=0.65\textwidth]{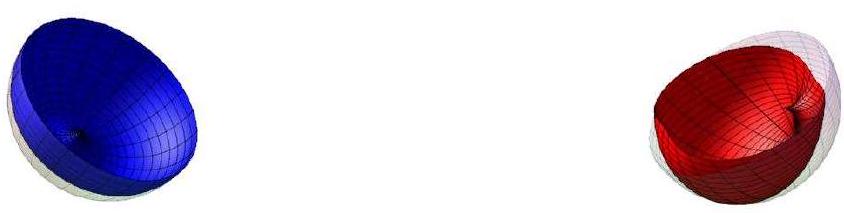}
\end{center}
\textbf{Figure 1.1:} \emph{Caustic of the hemi-sphere when the rays are parallel to the hemi-sphere axis of symmetry (left) and non parallel to this axis (right)}\\[2ex]

\textbf{Example 1.4.} General surface of revolution. We have

$$
\mathbf{r}=\{x(u) \cos (v), x(u) \sin (v), z(u)\}, \quad\left(\left(x^{\prime}\right)^{2}+\left(z^{\prime}\right)^{2}=1\right)
$$
and let $\mathbf{a}=\{0,0,1\}$. Then
$$
\begin{aligned}
& \mathbf{e}_{1}=\left\{x^{\prime} \cos v, x^{\prime} \sin v, z^{\prime}\right\}, \\
& \mathbf{e}_{2}=\{-\sin v, \cos v, 0\}, \\
& \mathbf{e}_{3}=\left\{-z^{\prime} \cos v,-z^{\prime} \sin v, x^{\prime}\right\}
\end{aligned}
$$
and hence $\mathbf{a}=z^{\prime} \mathbf{e}_{1}+x^{\prime} \mathbf{e}_{3}$, i.e. $a_{1}=z^{\prime}, a_{2}=0, a_{3}=x^{\prime}$. The equation on $k^{*}$ takes the form
$$
\left(k^{*}\right)^{2}+2\left(\frac{k_{1}}{a_{3}}+a_{3} k_{2}\right) k^{*}+4 k_{1} k_{2}=0
$$
with evident solutions $k_{1}^{*}=-\frac{2 k_{1}}{a_{3}}, \quad k_{2}^{*}=-2 a_{3} k_{2}$. The caustics are
$$
\boldsymbol\xi_{1}^{*}=r-\frac{a_{3}}{2 k_{1}} \mathbf{b}, \quad \boldsymbol\xi_{2}^{*}=r-\frac{1}{2 k_{2} a_{3}} \mathbf{b}
$$
where
$$
\mathbf{b}=\mathbf{a}-2 a_{3} \mathbf{n}=z^{\prime} \mathbf{e}_{1}-x^{\prime} \mathbf{e}_{3}=\left\{2 x^{\prime} z^{\prime} \cos v, 2 x^{\prime} z^{\prime} \sin v,\left(z^{\prime}\right)^{2}-\left(x^{\prime}\right)^{2}\right\} .
$$
So we have
$$
\boldsymbol\xi_{1}^{*}=\left\{\left(x-\frac{\left(x^{\prime}\right)^{2} z^{\prime}}{\left(z^{\prime \prime} x^{\prime}-z^{\prime} x^{\prime \prime}\right)}\right) \cos v,\left(x-\frac{\left(x^{\prime}\right)^{2} z^{\prime}}{\left(z^{\prime \prime} x^{\prime}-z^{\prime} x^{\prime \prime}\right)}\right) \sin v, z-\frac{x^{\prime}\left(\left(z^{\prime}\right)^{2}-\left(x^{\prime}\right)^{2}\right)}{2\left(z^{\prime \prime} x^{\prime}-z^{\prime} x^{\prime \prime}\right)}\right\}, $$
$$
\boldsymbol \xi_{2}^{*}=\left\{0,0, z-\frac{x\left(\left(z^{\prime}\right)^{2}-\left(x^{\prime}\right)^{2}\right)}{2 x^{\prime} z^{\prime}}\right\}
$$
The $\boldsymbol\xi_{1}^{*}$ is a surface of revolution generated by caustic of curve on a plane in case when incident rays are parallel to the axis of revolution.
The caustic $\boldsymbol\xi_{2}^{*}$ is the degenerated one.\\[1ex]

\textbf{Example 1.5. }Translation surface. Let us be given a translation surface $\mathbf{r}=\{x, y, f(x)+h(y)\}$ and suppose $\mathbf{a}=\{0,0,1\}$. Then
$$
\mathbf{n}=-\frac{\partial_{x} \mathbf{r} \times \partial_{y} \mathbf{r}}{\left|\partial_{x} \mathbf{r} \times \partial_{y} \mathbf{r}\right|}=\frac{1}{\sqrt{1+f_{x}^{2}+h_{y}^{2}}}\left\{f_{x}, h_{y},-1\right\}
$$
and
$$
\mathbf{b}=\frac{1}{1+f_{x}^{2}+h_{y}^{2}}\left\{2 f_{x}, 2 h_{y},-1+f_{x}^{2}+h_{y}^{2}\right\} .
$$
A direct computation show that in this case
$$
g^{*}=\left(\begin{array}{ll}
1 & 0 \\
0 & 1
\end{array}\right), \quad B^{*}=\frac{2}{1+f_{x}^{2}+h_{y}^{2}}\left(\begin{array}{cc}
-f_{x x} & 0 \\
0 & -h_{y y}
\end{array}\right)
$$
and hence
$$
\frac{1}{k_{1}^{*}}=-\frac{1+f_{x}^{2}+h_{y}^{2}}{2 f_{x x}}, \quad \frac{1}{k_{2}^{*}}=-\frac{1+f_{x}^{2}+h_{y}^{2}}{2 h_{y y}}
$$
Therefore, the equations of caustics take the forms
$$
\begin{aligned}
& \boldsymbol\xi_{1}^{*}=\{x, y, f(x)+h(y)\}-\frac{1}{2 f_{x x}}\left\{2 f_{x}, 2 h_{y},-1+f_{x}^{2}+h_{y}^{2}\right\} \\
&\boldsymbol \xi_{2}^{*}=\{x, y, f(x)+h(y)\}-\frac{1}{2 h_{y y}}\left\{2 f_{x}, 2 h_{y},-1+f_{x}^{2}+h_{y}^{2}\right\}
\end{aligned}
$$
In partial case of the hyperbolic paraboloid $\mathbf{r}=\left\{x, y, \frac{1}{2} x^{2}-\frac{1}{2} y^{2}\right\}$ we have $f_{x}=x, f_{x x}=1, h_{y}=-y, h_{y y}=-1$ and after simplifications we get
$$
\boldsymbol\xi_{1}^{*}=\left\{0,2 y, \frac{1}{2}-y^{2}\right\}, \quad \boldsymbol\xi_{1}^{*}=\left\{2 x, 0, x^{2}-\frac{1}{2}\right\}
$$
The caustics degenerate into two parabolas (See Figure 1.2). In partial case of elliptic paraboloid $\mathbf{r}=\left\{x, y, \frac{1}{2} x^{2}+\frac{1}{2} y^{2}\right\}$ the both caustics degenerate into one point $\left(0,0, \frac{1}{2}\right)$ (see Figure 1.3 for the other cases)\\
\includegraphics[max width=0.65\textwidth, center]{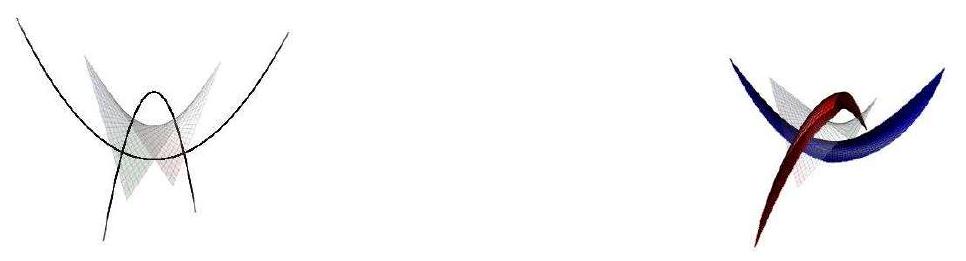}

\textbf{Figure 1.2:}\emph{ Reflected caustics of hyperbolic paraboloid in case of incident rays parallel to the axis of symmetry (left) and non-parallel to this axis (right)}\\[1ex]

\begin{center}
\includegraphics[max width=0.65\textwidth]{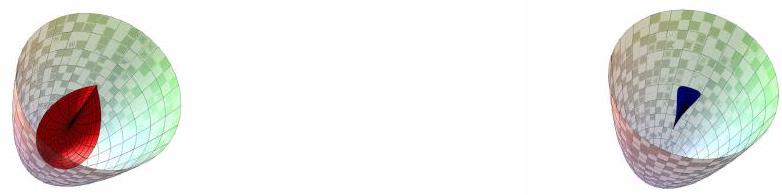}
\end{center}
\textbf{Figure 1.3:} \emph{Two reflected caustics of elliptic paraboloid in case of incident rays parallel non-parallel to the axis of symmetry
respectively.} \\[1ex]

\textbf{Example 1.6.} Cylinder over curve in a plane. Consider a cylinder based on a naturally parameterized curve $\gamma(s)$ in $x O y$ plane and suppose the rulings are directed along the $O z$ axis. Then the parametric equation of the cylinder is $\mathbf{r}=\gamma+t \mathbf{e}_{3}$. Denote by $\boldsymbol\tau=\gamma_{s}^{\prime}$ and $\boldsymbol\nu$ the Frenet frame of the curve. The orthonormal tangent frame of the surface consists of $\boldsymbol\tau$ and $\mathbf{e}_{3}$. The unit normal vector field of the surface is $\boldsymbol\nu$. Suppose $\mathbf{a} \perp \mathbf{e}_{3}$. Then the decomposition of $\mathbf{a}$ with respect to the frame $\boldsymbol\tau, \mathbf{e}_{3}, \boldsymbol\nu$ takes the form
$$
\mathbf{a}=\sin \theta \boldsymbol\tau+\cos \theta \boldsymbol\nu.
$$
The principal curvatures of this kind of cylinder are $k_{1}=k(s)$ and $k_{2}=0$, where $k(s)$ is the curvature of $\gamma$. Hence $K=0,2 H=k(s)$. In addition, the first and second fundamental forms are
$$
g=\left(\begin{array}{ll}
1 & 0 \\
0 & 1
\end{array}\right), \quad B=\left(\begin{array}{cc}
k(s) & 0 \\
0 & 0
\end{array}\right)
$$
Thus we have the following equation on $k^{*}$
$$
\left(k^{*}\right)^{2}+2\left(k(s) \cos \theta+\frac{1}{\cos \theta} k(s) \sin ^{2} \theta\right) k^{*}=0 .
$$
We have two solutions\\
(1) $k^{*}=0$,\\
(2) $k^{*}=-\frac{2 k(s)}{\cos \theta}$.

The solution (1) gives rise to "caustic" at infinity. The solution (2) gives rise to cylinder over caustic of the curve $\gamma$
$$
\boldsymbol\xi^{*}=\mathbf{r}-\frac{\cos \theta}{2 k(s)}(\sin \theta \boldsymbol\tau-\cos \theta \boldsymbol \nu)=\mathbf{r}+\frac{(\mathbf{a}, \boldsymbol \nu)}{2 k(s)}(-(\mathbf{a}, \boldsymbol\tau) \boldsymbol\tau+(\mathbf{a},\boldsymbol \nu) \boldsymbol \nu),
$$
which is the same as one can find in textbooks (see, e.g. [10], p. 109).

\section{ Caustics by reflection in case of spherical incident front.}
In this section we consider the case when the source of the rays is located at the origin.\\
\begin{theorem} Let $S$ be a regular surface parameterized by position-vector $\mathbf{r}$ and $\mathbf{n}$ be a the unit normal vector field of $S$. Denote by $\mathbf{a}=\frac{\mathbf{r}}{r}(r=|\mathbf{r}|)$ a direction of incident rays. Denote by $k_{i}^{*}(i=1,2)$ the solutions of
$$
\left(k^{*}+\frac{1}{r}\right)^{2}+2 \cos \theta\left(2 H+k_{n}\left(a_{t}\right) \tan ^{2} \theta\right)\left(k^{*}+\frac{1}{r}\right)+4 K=0
$$
where $H$ is the mean curvature, $K$ is the Gaussian curvature and $k_{n}\left(a_{t}\right)$ is the normal curvature of the surface in a direction of tangential projection of the incident rays. Then, over each local domain where $k_{i}^{*} \neq 0$, the parametric equations of caustics of the reflected front can be given by
$$
\boldsymbol\xi_{i}^{*}=\mathbf{r}+\frac{1}{k_{i}^{*}} \mathbf{b}
$$
where $\mathbf{b}=\mathbf{a}-2(\mathbf{a}, \mathbf{n}) \mathbf{n}$ is a direction of the reflected rays.\\
\end{theorem}

\begin{proof} By Lemma 1.1, we can restrict our research to the case $\lambda=0$. Denote $r=|\mathbf{r}|$. After computations we get
\begin{equation*}
g_{i j}^{*}=g_{i j}-\left(\partial_{j} \mathbf{r}, \mathbf{a}\right)\left(\partial_{i} \mathbf{r}, \mathbf{a}\right), \  B_{i j}^{*}=-2 \cos \theta B_{i j}-\left(\partial_{i} \mathbf{r}, \partial_{j} \mathbf{b}\right)=-2 \cos \theta B_{i j}-\frac{1}{r} g_{i j}^{*} \tag{2.1}
\end{equation*}
This means that in this case the Weingarten matrix $W^{*}$ takes the form
$$
W^{*}=A^{*}-\frac{1}{r} E
$$
where $A^{*}$ is a matrix of the same structure as the Weingarten matrix of the flat front. The equation on $k^{*}$ take the form
$$
\operatorname{det}\left(A^{*}-\left(k^{*}+\frac{1}{r}\right) E\right)=0
$$
or
\begin{equation*}
\left(k^{*}+\frac{1}{r}\right)^{2}+2 \cos \theta\left(2 H+k_{n}\left(a_{t}\right) \tan ^{2} \theta\right)\left(k^{*}+\frac{1}{r}\right)+4 K=0 \tag{2.2}
\end{equation*}
If we denote by $k_{i}^{*}$ the solution for (2.2), then the parametric equation of the caustic takes the form
$$
\boldsymbol\xi_{i}^{*}=\mathbf{r}+\frac{1}{k_{i}^{*}} \mathbf{b}
$$
provided that $k_{i}^{*} \neq 0$ over a local domain.
\end{proof}

\textbf{Example 2.2. }The sphere of radius 1 centered at the origin. Suppose the emitting point is at the origin and hence

$$
r=1, \mathbf{a}=\mathbf{r}=-\mathbf{n}, \quad \theta=\pi, H=K=1
$$

Then the equation on $k^{*}$ takes the form $\left(k^{*}+1\right)^{2}-4\left(k^{*}+1\right)+4=0$, i.e. $k^{*}=1$ and hence $\xi^{*}=\mathbf{r}+\mathbf{n}=0$.\\
\includegraphics[max width=0.65\textwidth, center]{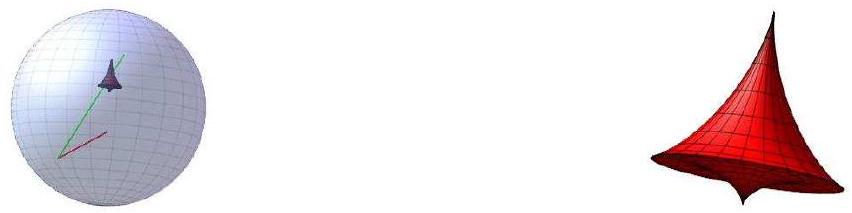}

\textbf{Figure 2.1:}\emph{ Caustic inside the sphere when the source is located not far from the center (left). This caustic itself (right)}\\

\includegraphics[max width=0.65\textwidth, center]{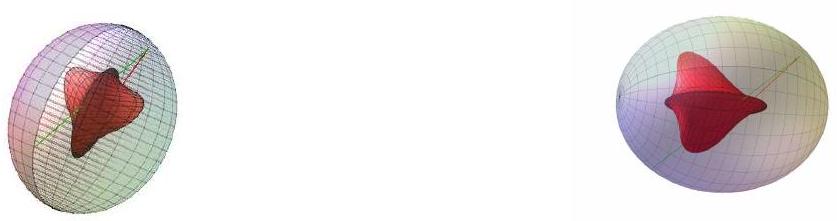}

\textbf{Figure 2.2:} \emph{Caustic inside the three-axis ellipsoid when the source is located not far from the center of symmetry}\\[1ex]

If the source does not located at the center, the caustic of reflected front has rather complicated structure, see Figure 2.1.  Inside three-axis ellipsoid in case of source does not located at the center of symmetry (but not too far from it), the caustic can be seen on Figure 2.2.\\[1ex]

\textbf{Example 2.3. } Cylinder over curve in a plane. Consider a cylinder based on a naturally parameterized curve $\gamma({s})$ in $x O y$ plane and suppose the rulings are directed along the $O z$ axis. Then the parametric equation of the cylinder is $\mathbf{r}=\gamma+t \mathbf{e}_{3}$. Denote by $\boldsymbol\tau=\gamma_{s}^{\prime}$ and $\boldsymbol\nu$ the Frenet frame of the curve. The orthonormal tangent frame of the surface consists of $\boldsymbol\tau$ and $\mathbf{e}_{3}$. The unit normal vector field of the surface is $\boldsymbol\nu$. Denote $r=|\mathbf{r}|$ and $\mathbf{a}=\frac{\mathbf{r}}{r}$. Then the decomposition of $\mathbf{a}$ with respect to the frame $\boldsymbol\tau, \mathbf{e}_{3}, \boldsymbol\nu$ takes the form
$$
\mathbf{a}=\sin \theta \cos \alpha \boldsymbol\tau+\sin \theta \sin \alpha \, \mathbf{e}_3+\cos \theta \boldsymbol\nu
$$
As it well known, the principal curvatures of this kind of a cylinder are $k_{1}=k(s)$ and $k_{2}=0$, where $k(s)$ is the curvature of $\gamma$. Hence $K=0,2 H=k(s)$. In addition, the first and second fundamental forms are
$$
g=\left(\begin{array}{cc}
1 & 0 \\
0 & 1
\end{array}\right) \quad \text { and } \quad B=\left(\begin{array}{cc}
k(s) & 0 \\
0 & 0
\end{array}\right)
$$
respectively. Thus we have the equation on $k^{*}$ of the following form
$$
\left(k^{*}+\frac{1}{r}\right)^{2}+2\left(k(s) \cos \theta+\frac{1}{\cos \theta} k(s) \sin ^{2} \theta \cos ^{2} \alpha\right)\left(k^{*}+\frac{1}{r}\right)=0
$$
We have two solutions
$$
\text { (1) } k^{*}=-\frac{1}{r} \quad \text { (2) } k^{*}=-\frac{1}{r}-\frac{2 k(s)}{\cos \theta}\left(\cos ^{2} \theta+\sin ^{2} \theta \cos ^{2} \alpha\right)
$$
The cross-section $\alpha=0$ is the caustic of rays reflected by the curve $\gamma$ in $x O y$-plane. The equation takes the form
$$
\boldsymbol\xi^{*}=\mathbf{r}-\frac{\cos \theta}{\frac{\cos \theta}{r}+2 k(s)}(\sin \theta \boldsymbol\tau-\cos \theta\boldsymbol\nu)=\mathbf{r}+\frac{(\mathbf{r}, v)}{(\mathbf{r}, v)+2 r^{2} k(s)}(-(\mathbf{r}, \boldsymbol\tau) \boldsymbol\tau+(\mathbf{r},\boldsymbol\nu)\boldsymbol\nu)
$$
which is the same as one can find in textbooks (see, e.g. \cite{10}, p. 109).

\section{Conclusion}
In this paper we have found exact and simple parameterization of caustics of reflected wave fronts by using purely geometric approach. Now we know which geometric characteristic of a mirror surface define geometry of the caustic, how the location of the surface with respect to incident pencil of rays changes the caustic. As a byproduct, the results allow to get fast and exact formulas for computer simulation of caustics of reflected fronts that can be used both in geometry and engineering. The examples approve the calculations and have nice visual representations.


\begin{thebibliography}{11}
\bibitem{1} J. W. Bruce, P. J. Giblin, C. G. Gibson. On caustics by reflexion, Topology 21(2) (1982), 179-199.
\bibitem{2} Chr. Georgiou, Th. Hasanis, D. Koutroufiotis. On the caustic of a convex mirror. Geometriae Dedicata 28 (1988), 153-169
\bibitem{3} S. Izumiya, Perestroikas of optical wave fronts and graphlike Legendrian unfoldings. J. Differential Geom. 38 (1993) 485-00.
\bibitem{4} S. Izumiya, M. Takahashi. On caustics of submanifolds and canal hypersurfaces in Euclidean space. Topology and its Applications, 159 (2012) 501 508.
\bibitem{5} S. Izumiya, M. Takahashi, Caustics and wave front propagations: Applications to differential geometry, in: Geometry and Topology of Caustics. Banach Center Publ. 82 (2008), 125-142
\bibitem{6} J. Chen, H. Liu, J. Miao. Caustics of translation surfaces in Euclidean 3-space. Nonlinear Sci. Appl., 10 (2017), 5300 - 5310. doi:10.22436/jnsa.010.10.16
\bibitem{7} G. Glaeser. Refections on spheres and cylinders of revolution. Journal for Geometry and Graphics. 3(2) (1999), 121 - 139.
\bibitem{8} D. R. J. Chillingworth, G. R. Danesh-Narouie, B. S. Westcott. On Ray-Tracing Via Caustic Geometry. IEEE transactions on antennas and propagation 38(5) May 1990, 625-632
\bibitem{9} M. Kokubu, W. Rossman, M. Umehara and K. Yamada. Flat fronts in hyperbolic 3-space and their caustics, J. Math. Soc. Japan 59(1) (2007), $265-299$
\bibitem{10} Rovenski V.: Geometry of Curves and Surfaces with MAPLE. Birkhäuser Boston, 2000.
\end{thebibliography}
\end{document}